%% file: PrimalDPGparabolic.tex
\documentclass[11pt]{amsart}

\usepackage[utf8]{inputenc}
\usepackage[T1]{fontenc}

\usepackage{amsmath,amssymb,amsfonts,textcomp,amsthm,xifthen,graphicx,color,pgfplots}
\usepackage{ifthen}

\usepackage{enumerate}
\usepackage{enumitem}

\usepackage{fullpage}

\usepackage[utf8]{inputenc}

\usepackage{float} 
\usepackage{array}

\usepackage{multirow}

\usepackage{etoolbox}
\newcounter{rowcntr}[table]
\renewcommand{\therowcntr}{\alph{rowcntr})}

\newcolumntype{N}{>{\refstepcounter{rowcntr}\therowcntr}c}

\AtBeginEnvironment{tabular}{\setcounter{rowcntr}{0}}

\input{header}
\begin{document}

\title[]{Analysis of backward Euler primal DPG methods}
\date{\today}
\author{Thomas F\"{u}hrer}
\address{Facultad de Matem\'{a}ticas, Pontificia Universidad Cat\'{o}lica de Chile, Santiago, Chile}
\email{tofuhrer@mat.uc.cl}

\author{Norbert Heuer}
\address{Facultad de Matem\'{a}ticas, Pontificia Universidad Cat\'{o}lica de Chile, Santiago, Chile}
\email{nheuer@mat.uc.cl}

\author{Michael Karkulik}
\address{Departamento de Matem\'{a}tica, Universidad T\'{e}cnica Federico Santa Mar\'{i}a, Valpara\'{i}so, Chile}
\email{michael.karkulik@usm.cl}

\thanks{{\bf Acknowledgment.} 
This work was supported by ANID through FONDECYT projects 1210391 and 1190009 and 1210579.}

\begin{abstract}
We analyse backward Euler time stepping schemes for the primal DPG formulation of a class of parabolic problems.
Optimal error estimates are shown in the natural norm and in the $L^2$ norm of the field variable. 
For the heat equation the solution of our primal DPG formulation equals the solution of a standard Galerkin scheme and, thus, optimal error bounds are found in the literature. 
In the presence of advection and reaction terms, however, the latter identity is not valid anymore and the analysis of optimal error bounds requires to resort to elliptic projection operators.
It is essential that these operators be projections with respect to the spatial part of the PDE, as in standard Galerkin schemes, and not with respect to the full PDE at a time step, as done previously.
\end{abstract}
\maketitle

\section{Introduction}
In this work we analyse a backward Euler primal DPG time stepping scheme for the parabolic problem
\begin{subequations}
\begin{alignat}{2}
  \dot{u}-\div\AA \nabla u + \bbeta\cdot\nabla u + \gamma u &=f &\quad&\text{in } (0,T)\times\Omega, \label{eq:model:pde} \\
  u &= 0 &\quad&\text{in } (0,T)\times \partial\Omega, \\
  u(0,\cdot) &= u_0 &\quad&\text{in } \Omega.
\end{alignat}
\end{subequations}
Here, $\Omega\subset\R^d$ ($d\in 2,3$) denotes a bounded Lipschitz domain.
We make the standing assumptions that the coefficients $\AA,\bbeta$, and $\gamma$ are time-independent,
that they are essentially bounded in space, that $\AA$ is a symmetric positive matrix almost everywhere in $\Omega$,
and that $\ip{\bbeta\cdot\nabla v+\gamma v}{v} \geq 0$ for all $v\in H^1_0(\Omega)$. 
The $L^2(\Omega)$ inner product is denoted with $\ip\cdot\cdot$.
We assume that the data satisfy $f(t,\cdot)\in L^2(\Omega)$ for all $t\in[0,T]$ and $u_0\in L^2(\Omega)$.

The discontinuous Petrov--Galerkin method with optimal test functions (DPG) pertains to the class of minimum residual
methods and was introduced in a series of papers~\cite{partI,partII,partIII}.
It has been successfully applied to elliptic problems, see, e.g.,~\cite{DemkowiczG_11_ADM,DemkowiczG_13_PDM} for the Poisson
problem and~\cite{KLove1,DPGbiLaplace} for fourth-order problems.
Through the use of optimal test functions, the discrete problem inherits the stability of the continuous problem.
This comes in advantageous for problems where robustness is one of the main challenges, e.g.,
singularly perturbed problems, see~\cite{HeuerK_2017} for reaction-dominated diffusion problems
and~\cite{DemkowiczHeuer_2013,BroersenStevenson_2014} for the convection-dominated case.
Space-time DPG methods have been studied previously, see, e.g.,~\cite{DGNS_Schroedinger,Wieners19,DPGspaceTime}.
For other space-time minimum residual methods we refer to~\cite{LSQparabolicSpaceTime,Andreev_13,VassilevskiEtAl18,StevensonWesterdiep19}.
Approaches employing the DPG methodology for the time discretization of parabolic and hyperbolic initial value problems have recently
been investigated, cf.~\cite{MunozPD_21,MunozPD_21a}.
On the other hand, time-stepping methods for ODEs are frequently employed in combination with standard Galerkin finite element methods in
space, cf. the monograph~\cite{thomee} for parabolic equations, but less so with DPG methods.
To the best of our knowledge there exist only two works in this direction, dealing with time-stepping and spatial DPG methods for the heat equation,
namely~\cite{DPGtimestepping} and~\cite{RobertsHenneking2021}.

In~\cite{DPGtimestepping}, a backward Euler method is used to discretise in time, and then the DPG methodology is applied to the ultraweak
variational formulation of the resulting equations. The a priori analysis given there employs the Galerkin projection with respect to these very
equations, and hence a spatial discretization error has to be accounted for in every time step. This gives rise to a theoretical error bound of order
$\OO(h/k) + \OO(k)$ for lowest-order discretizations with $h$ and $k$ being the spatial mesh-width and time step, respectively. As numerical experiments from~\cite{RobertsHenneking2021} indicate, this asymptotic error bound is not optimal. 
The authors of~\cite{RobertsHenneking2021} study general $\theta$-schemes (including the backward Euler and the Crank--Nicolson time discretization)
based on the primal DPG method~\cite{DemkowiczG_13_PDM} and the ultraweak DPG method~\cite{DemkowiczG_11_ADM}, and provide an
extensive numerical study and comparison of the different approaches, and it turns out that $\OO(h+k)$ is the optimal error bound which can
be expected for the method from~\cite{DPGtimestepping}.

Our motivation for the present work is to give a theoretically sound explanation of the optimal convergence rates
seen in the numerical experiments from~\cite{RobertsHenneking2021} for the backward Euler primal DPG method.
We will consider general second order linear elliptic spatial differential operators.
For the heat equation ($\AA$ is the identity, $\bbeta=0$, and $\gamma=0$) optimal error estimates follow from the fact that the
field solution component of the primal DPG method is identical to the solution of
a standard Galerkin FEM (cf. Section~\ref{sec:remheateq} for details). 
Therefore, well-known results for time-stepping Galerkin FEMs, see e.g.~\cite{thomee}, apply.
However, for the general case $\bbeta\neq 0$ this is not true anymore and a new analysis needs to be provided.
The accumulation of the spatial discretization error in every time step, identified above as the reason for suboptimality of theoretical results,
is usually avoided in standard FEMs by using the Galerkin projection with respect to the elliptic part of the parabolic equation only.
This idea, nowadays being referred to as \textit{elliptic projection operators}, was introduced in~\cite{Wheeler73} and is by now one of the main
tools in the analysis of time-stepping FEMs, as witnessed again in~\cite{thomee}.
In the present work we prove optimal error estimates in the context of practical primal DPG methods by the use of an elliptic projection operator.
We follow some ideas from~\cite{LSQparabolic} where an elliptic projection in the analysis of time-stepping first-order
system least-squares finite element methods is used in the same spirit.

The remainder of this paper is organised as follows: 
In Section~\ref{sec:main} we introduce the fully discrete method as well as the necessary notation. We prove stability
of the method and provide quasi-optimality results for the elliptic projection operator.
In Section~\ref{sec:strong} we use these results to show optimal error estimates in the $H^1(\Omega)$ norm.
Section~\ref{sec:weak} is devoted to optimal error estimates in weaker norms, particularly in $L^2(\Omega)$.

\section{Time-stepping DPG formulation}\label{sec:main}

\subsection{Notation}
The notation $a\lesssim b$ means that there exists a constant $C>0$ that (possibly) depends on $\Omega$, $\AA$, $\bbeta$, $\gamma$, and $T$, but is independent of involved functions.
We write $a\eqsim b$ if $a\lesssim b$ and $b\lesssim a$. Furthermore, $a\gtrsim b$ stands for $b\lesssim a$.

We consider a time discretization $0 = t_0 < \cdots < t_N=T$, and for notational simplicity
we use a uniform time step size $k = t_n-t_{n-1}$ (but we stress that this is not necessary).
By $\ip{\cdot}{\cdot}$ and $\norm{\cdot}{}$ we denote the inner product and norm in $L^2(\Omega)$.
We consider spatial discretizations based on a shape-regular conforming simplicial mesh $\TT$ of $\Omega$.
To any partition $\TT$ we associate its skeleton $\cS$ consisting of the boundaries of all elements, and the trace space
\begin{align*}
  H^{-1/2}(\cS) := \left\lbrace \widehat\sigma \in \prod_{K\in\TT} H^{-1/2}(\partial K) \middle\vert \exists \ssigma \in \Hdivset{\Omega}:
  \widehat\sigma|_{\partial K} = \ssigma\cdot\nn_K|_{\partial K}, \quad\forall K\in\TT \right\rbrace,
\end{align*}
where $\nn_K$ denotes the unit outward normal vector on $\partial K$. This is a Hilbert space with norm
\begin{align*}
  \norm{\widehat\sigma}{-1/2,k} := \inf\left\lbrace(\norm{\ssigma}{}^{2}+k\norm{\div\ssigma}{}^2)^{1/2}\middle\vert 
  \ssigma\in\Hdivset\Omega, \widehat\sigma|_{\partial K} = \ssigma\cdot\nn_K|_{\partial K}, \quad\forall K\in\TT \right\rbrace.
\end{align*}
The trial space of our method will be
\begin{align*}
  U := H_0^1(\Omega)\times H^{-1/2}(\cS)
\end{align*}
equipped with the norm
\begin{align*}
  \norm{(u,\widehat\sigma)}{U,k}^2 := \norm{\nabla u}{}^2 + \norm{\widehat\sigma}{-1/2,k}^2,
\end{align*}
and the test space will be
\begin{align*}
  V := H^1(\TT) := \prod_{K\in\TT} H^1(K)
\end{align*}
equipped with the norm
\begin{align*}
  \norm{v}{V,k}^2 := \frac1{k}\norm{v}{}^2 + \norm{\AA^{1/2}\pwnabla v}{}^2.
\end{align*}
Here, $\pwnabla v$ is the $\TT$-piecewise gradient.
For functions $\widehat\sigma\in H^{-1/2}(\cS)$ we define
\begin{align*}
  \dual{\widehat\sigma}{v}_{\cS} := \ip{\div\ssigma}{v}+\ip{\ssigma}{\pwnabla v} \quad\forall v\in H^1(\TT),
\end{align*}
where $\ssigma\in \Hdivset\Omega$ with $\ssigma\cdot\normal_K|_{\partial K} = \widehat\sigma|_{\partial K}$ for all $K\in\TT$.
We note that the definition is independent of the choice of $\ssigma$ and recall the following result.
\begin{lemma}\label{lem:traceidentity}
  \begin{align*}
    \norm{\widehat\sigma}{-1/2,k} \eqsim \sup_{0\neq v\in V} \frac{\dual{\widehat\sigma}{v}_{\cS}}{\norm{v}{V,k}}.
  \end{align*}
\end{lemma}
\begin{proof}
  The proof follows by now well-established arguments, see~\cite{breakSpace} or~\cite{DPGlargeDomain} for problems where the norms depend on parameters.
  Note that if we use $\norm{A^{-1/2}\ssigma}{}$ instead of $\norm{\ssigma}{}^2$ in the definition of $\norm{\cdot}{-1/2,k}$ then equality holds.
\end{proof}

\subsection{Primal DPG formulation}\label{sec:dpg}

For a simpler notation, we use superindices to indicate time evaluations, i.e., for a time-dependent function $v =
v(t;\cdot)$ we use the notation $v^n(\cdot) := v(t_n;\cdot)$.
Moreover, we use similar notations for discrete quantities which however are not defined for all $t\in [0,T]$. 
For instance, later on $u_h^n$ will denote an approximation of $u^n = u(t_n;\cdot)$ where $u$ is the exact solution of
the parabolic equation.
We approximate $\dot{u}^n$ by a backward difference, i.e., $\dot{u}^n \approx (u^n-u^{n-1})/k$,
this formally yields the elliptic PDE
\begin{align*}
  \frac{1}k u^n - \div\AA\nabla u^n + \bbeta\cdot\nabla u^n + \gamma u^n &= f^n + \frac1k u^{n-1},
\end{align*}
which admits a unique solution $u^n\in H_0^1(\Omega)$.
By testing with $v\in V$ and introducing $\widehat\sigma^n$ as
\begin{align*}
  \widehat\sigma^n|_{\partial K} := \AA\nabla u^n\cdot\nn_K|_{\partial K}, \quad\forall K\in\TT,
\end{align*}
and integrating by parts one obtains the primal DPG formulation
\begin{align*}
  \frac1k\ip{u^n}v + \ip{\AA\nabla u^n}{\pwnabla v} + \ip{\bbeta\cdot\nabla u^n}{v} + \ip{\gamma u^n}{v} - \dual{\widehat \sigma^n}{v}_{\cS}
  = \ip{f^n}{v} + \frac1k\ip{u^{n-1}}v.
\end{align*}

We introduce some bilinear forms and the right-hand side functional:
For $\uu=(u,\widehat\sigma)\in U$, $v\in V$ and $w,g\in L^2{(\Omega)}$ set
\begin{align*}
  b(\uu,v) &:= \ip{\AA\nabla u}{\pwnabla v} + \ip{\bbeta\cdot\nabla u}{v} + \ip{\gamma u}{v} - \dual{\widehat \sigma}{v}_{\cS}, \\
  a(\uu,v) &:= \frac1k\ip{u}v + b(\uu,v), \\
  F(g,w;v) &:= \ip{g}{v} + \frac1k\ip{w}{v}.
\end{align*}

Therefore, our formulation simply reads
\begin{align*}
  a(\uu^n,v) = F(f^n,u^{n-1};v) \quad\forall v\in V. 
\end{align*}

We note that the bilinear form $a(\cdot,\cdot)$ satisfies a continuous $\inf$--$\sup$ condition
using the norms $\norm{\cdot}{U,k}$ and $\norm{\cdot}{V,k}$. This can be seen by extending the analysis of
standard works like~\cite{breakSpace} to general elliptic PDEs. We do not pursue this further,
as we ultimately want to analyse the practical DPG method. Furthermore, we note that $a$ is not bounded independently of $k$
using the above norms, but rather
\begin{align}\label{eq:a:bdd}
  |a(\uu,v)| \lesssim \left( k^{-1/2}\norm{u}{} + \norm{\uu}{U,k} \right) \norm{v}{V,k},
\end{align}
where $\uu=(u,\widehat\sigma)$.

\subsubsection{Fully discrete scheme}
By $\PP^p(K)$ we denote the space of polynomials of degree $p\in\N_0$, and 
\begin{align*}
  \PP^p(\TT) := \set{v\in L^2(\Omega)}{v|_K\in \PP^p(K)\,\,\forall K\in\TT}
\end{align*}
For the discretization of traces we use the space of facewise polynomials denoted by $\PP^p(\cS)$. In particular, we note that $\PP^p(\cS)$ is the (normal-) trace space of the $p$-th order Raviart--Thomas space.
We consider the spaces
\begin{align*}
  U_h &:= \PP^{p+1}(\TT)\cap H_0^1(\Omega) \times \PP^p(\cS), \\
  V_h &:= \PP^{p+d}(\TT).
\end{align*}
The (discrete) trial-to-test operator is given by 
\begin{align*}
  \ip{\Theta_h\uu_h}{v_h}_{V,k} = a(\uu_h,v_h) \quad\forall v_h\in V_h.
\end{align*}
We recall that the inner product is given by
\begin{align*}
  \ip{v}{\delta v}_{V,k} = \frac1k\ip{v}{\delta v} + \ip{\AA\pwnabla v}{\pwnabla\delta v}.
\end{align*}

The fully discrete scheme then reads: Given $u_h^0\in L^2(\Omega)$, solve
\begin{align}\label{eq:euler}
  a(\uu_h^n,\Theta_h \ww_h) = F(f^n,u_h^{n-1};\Theta_h\ww_h) \quad\forall \ww_h\in U_h, \, n=1,2,\dots.
\end{align}

The next lemma establishes coercivity of $a(\cdot,\Theta_h\cdot)$ on $U_h$.

\begin{lemma}\label{lem:euler}
  Let $\uu_h = (u_h,\widehat\sigma_h)\in U_h$. Then, 
  \begin{align*}
    \frac1k\norm{u_h}{}^2 + \norm{\AA^{1/2}\nabla u_h}{}^2 \leq \norm{\Theta_h \uu_h}{V,k}^2.
  \end{align*}
  If additionally, $hk^{-1/2}\leq C_1$ for some constant $C_1>0$, then,
  \begin{align*}
    C_2\norm{\widehat\sigma_h}{-1/2,k} \leq C_3\norm{\Theta_h(0,\widehat\sigma_h)}{V,k} \leq \norm{\Theta_h\uu_h}{V,k}.
  \end{align*}
  Here, $C_2,C_3>0$ denote generic constants independent of $h,k$ and $\uu_h$. 
\end{lemma}
\begin{proof}
  For the first part we use that $\PP^{p+1}(\TT)\cap H_0^1(\Omega)\subset V_h$ and $\ip{\bbeta\cdot\nabla w+\gamma w}w\geq 0$ for $w\in H_0^1(\Omega)$: 
  Let $\uu_h=(u_h,\widehat\sigma_h)\in U_h$ be given. With the latter observation, the definition of the optimal test function and the fact that $\dual{\widehat\sigma_h}{w}_{\cS} = 0$ for $w\in H_0^1(\Omega)$ we get that
  \begin{align*}
    \frac{1}k\norm{u_h}{}^2 + \norm{\AA^{1/2}\nabla u_h}{}^2 &\leq 
    \frac{1}k\ip{u_h}{u_h} + \ip{\AA\nabla u_h}{\nabla u_h} + \ip{\bbeta\cdot\nabla u_h}{u_h} + \ip{\gamma u_h}{u_h}
    \\
    &= a( (u_h,0),u_h)= \ip{\Theta_h\uu_h}{u_h}_{V,k} \leq \norm{\Theta_h\uu_h}{V,k}\norm{u_h}{V,k} 
    \\&= \norm{\Theta_h\uu_h}{V,k}\left(\frac1k\norm{u_h}{}^2 + \norm{\AA^{1/2}\nabla u_h}{}^2\right)^{1/2}.
  \end{align*}
  This finishes the proof of the first estimate.

  For the second one we use some results established in the literature on fully discrete DPG formulations (practical
  DPG), e.g.~\cite{practicalDPG}.
  Let $\Pi_h\colon V\to V_h$ be the Fortin operator defined in~\cite[Lemma~3.2]{practicalDPG} which
  in~\cite{practicalDPG} is denoted by $\Pi_r^\mathrm{grad}$ with $r=p+d$.
  By Lemma~\ref{lem:traceidentity} and the Fortin property we have that
  \begin{align}\label{eq:euler:proof:a}
    \norm{\widehat\sigma_h}{-1/2,k} \eqsim \sup_{0\neq v\in H^1(\TT)} \frac{\dual{\widehat\sigma_h}v_{\cS}}{\norm{v}{V,k}} = 
    \sup_{0\neq v\in H^1(\TT)} \frac{\dual{\widehat\sigma_h}{\Pi_h v}_{\cS}}{\norm{v}{V,k}}.
  \end{align}
  Moreover, from~\cite[proof of Lemma~3.2]{practicalDPG} we infer that
  \begin{align*}
    k^{-1/2}\norm{\Pi_h v}{} + \norm{\nabla \Pi_h v}{} \lesssim k^{-1/2}\norm{v}{} + (1+h k^{-1/2})\norm{\nabla v}{}.
  \end{align*}
  Thus, using the assumption $hk^{-1/2}\lesssim 1$ we see that
  \begin{align*}
    \norm{\Pi_h v}{V,k} \lesssim \norm{v}{V,k} \quad\forall v\in V.
  \end{align*}
  Combining this with~\eqref{eq:euler:proof:a} we conclude that
  \begin{align*}
    \norm{\widehat\sigma_h}{-1/2,k} &\lesssim \sup_{0\neq v\in H^1(\TT)} \frac{\dual{\widehat\sigma_h}{\Pi_h v}_{\cS}}{\norm{v}{V,k}} \lesssim 
    \sup_{0\neq v\in H^1(\TT)} \frac{\dual{\widehat\sigma_h}{\Pi_h v}_{\cS}}{\norm{\Pi_h v}{V,k}}
    \\
    &\leq \sup_{0\neq v\in V_h} \frac{\dual{\widehat\sigma_h}{v_h}_{\cS}}{\norm{v_h}{V,k}} = \norm{\Theta_h(0,\widehat\sigma_h)}{V,k}.
  \end{align*}
  Then, $\norm{\Theta_h(0,\widehat\sigma_h)}{V,k}\leq \norm{\Theta_h(u_h,\widehat\sigma_h)}{V,k}+\norm{\Theta_h(u_h,0)}{V,k}$. It remains to estimate $\norm{\Theta_h(u_h,0)}{V,k}$: We use the notation $\Theta_h(u_h,0)=:v_h$, 
  \begin{align*}
    \norm{\Theta_h(u_h,0)}{V,k}^2 &= a((u_h,0),\Theta_h(u_h,0)) = \frac1k\ip{u_h}{v_h} + \ip{\AA\nabla u_h}{\pwnabla v_h} + \ip{\bbeta\cdot\nabla u_h+\gamma u_h}{v_h} 
    \\ &\lesssim (\frac1k\norm{u_h}{}^2+\norm{\AA^{1/2}\nabla u_h}{}^2)^{1/2}\norm{v_h}{V,k}
    \leq \norm{\Theta_h\uu_h}{V,k} \norm{\Theta_h(u_h,0)}{V,k}. 
  \end{align*}
  Note that the involved constant only depends on the end time $T$ and the coefficients, but is otherwise independent of $k$ or $h$. 
\end{proof}

It should be noted that the condition $hk^{-1/2}\lesssim 1$ is only needed to get an $k$ independent estimate for the
trace norm. If this condition is not satisfied then the constant depends on $k$.

\begin{theorem}
  Problem~\eqref{eq:euler} is well posed. In particular, the solutions are stable in the sense that
  \begin{align*}
    \norm{u_h^n}{} \leq \left(\norm{u_h^n}{}^2 + k\norm{\AA^{1/2}\nabla u_h^n}{}^2\right)^{1/2}\leq \sum_{j=1}^n k \norm{f^j}{} + \norm{u_h^0}{}.
  \end{align*}
\end{theorem}
\begin{proof}
  According to Lemma~\ref{lem:euler},
  Problem~\eqref{eq:euler} admits unique solutions $\uu_h^n\in U_h$, $n=1,2,\dots$.

  Then,
  \begin{align*}
    \norm{\Theta_h\uu_h^n}{V,k}^2 &= a(\uu_h^n,\Theta_h\uu_h^n) = F(f^n,u_h^{n-1};\Theta_h\uu_h^n)
    = \ip{f^n+k^{-1}u_h^{n-1}}{\Theta_h\uu_h^n} 
    \\
    &\leq k^{1/2}\norm{f^n+k^{-1}u_h^{n-1}}{}k^{-1/2}\norm{\Theta_h\uu_h^n}{} 
    \leq k^{1/2}\norm{f^n+k^{-1}u_h^{n-1}}{}\norm{\Theta_h\uu_h^n}{V,k}.
  \end{align*}
  This together with the estimate from Lemma~\ref{lem:euler} shows that
  \begin{align*}
   \left(\norm{u_h^n}{}^2 + k\norm{\AA^{1/2}\nabla u_h^n}{}^2\right)^{1/2}\leq k^{1/2}\norm{\Theta_h\uu_h^n}{V,k} \leq k \norm{f^n}{} + \norm{u_h^{n-1}}{}.
  \end{align*}
  Iterating the arguments concludes the proof.
\end{proof}

\subsubsection{Remark: Heat equation --- the trivial case}\label{sec:remheateq}
Let us consider the simplest model which is the heat equation where $\AA$ is the identity, $\bbeta=0$, $\gamma=0$. 

It is straightforward to see that for $(u_h,0)\in U_h$, $\Theta_h(u_h,0) = u_h$ since
\begin{align*}
  \ip{u_h}{v_h}_{V,k} = \frac1{k} \ip{u_h}{v_h} + \ip{\nabla u_h}{\pwnabla v_h} = a(\,(u_h,0),v_h) = \ip{\Theta_h(u_h,0)}{v_h}_{V,k}.
\end{align*}
Using the test functions $\Theta_h(v_h,0)$ in~\eqref{eq:euler} and the fact that $\dual{\widehat\sigma_h}{v_h}_{\cS}=0$ we see that
\begin{align*}
  a(\uu_h^n,\Theta_h(v_h,0)) = \frac1{k}\ip{u_h^n}{v_h} + \ip{\nabla u_h^n}{\nabla v_h}  = \ip{f^n+k^{-1}u_h^{n-1}}{v_h} = F(f^n,u_h^{n-1};\Theta_h(v_h,0))
\end{align*}
for all $v_h\in \PP^{p+1}(\TT)\cap H_0^1(\Omega)$.
Let $\uu_h^n=(u_h^n,\widehat\sigma_h^n)\in U_h$ denote the solution to~\eqref{eq:euler}, then the solution component $u_h$ satisfies
\begin{align*}
  \frac1{k}\ip{u_h^n}{v_h} + \ip{\nabla u_h^n}{\nabla v_h} = \ip{f^n}{v_h} + \frac1{k} \ip{u_h^{n-1}}{v_h} \quad\forall v_h\in \PP^{p+1}(\TT)\cap H_0^1(\Omega),
\end{align*}
which is the standard Galerkin FEM, see~\cite[Section~1]{thomee}. Thus the primal DPG solution component $u_h$ is identical to the standard Galerkin FEM solution. 
In particular, optimal error estimates in $L^\infty(L^2)$ and $L^\infty(H_0^1)$, see~\cite{thomee}, are valid for the primal DPG solution.
In~\cite{RobertsHenneking2021} optimal $L^\infty(H_0^1)$ have been observed in numerical experiments, which can be explained with the observation above.

Note that this remark is true only if we consider $\bbeta=0$, $\gamma=0$.

\subsection{Elliptic projection-type operator}
To obtain optimal error estimates we introduce an elliptic projection.
The idea goes back to~\cite{Wheeler73} to obtain optimal $L^2(\Omega)$ a priori error estimates and is extensively used for Galerkin methods,
but has not been studied for least-squares methods until only very recently in~\cite{LSQparabolic}. 
For DPG methods an additional difficulty arises since test norms are mesh dependent.
A main difference to Galerkin methods is that, although the elliptic part of the parabolic PDE might be symmetric,
our elliptic projection operator always corresponds to a non-symmetric problem. 
The reason is that we need to use optimal test functions used in~\eqref{eq:euler} in combination with the bilinear form
$b(\cdot,\cdot)$, which models the elliptic part.
\bigskip

We define the elliptic projection operator $\EE_h:U\rightarrow U_h$ by
\begin{align}\label{eq:ellipticProjection}
  b(\EE_h\uu,\Theta_h\ww_h) = b(\uu,\Theta_h\ww_h) \quad\forall \ww_h\in U_h. 
\end{align}

In the proofs below and in some results we will use the (semi-)norm 
  \begin{align*}
    \enorm{\uu}_k^2 := \norm{\AA^{1/2}\nabla u}{}^2 + \norm{\Theta_h(0,\widehat\sigma)}{V,k}^2.
  \end{align*}
Note that from the results of Lemma~\ref{lem:euler} we infer that $\enorm{\uu_h}_k \eqsim \norm{\uu_h}{U,k}$ for $\uu_h\in U_h$
and the equivalence constants are independent of $h,k$ if $hk^{-1/2}=\OO(1)$.
The next lemma establishes boundedness and an inf-sup condition of $b(\cdot,\Theta_h\cdot)$.
Note that this is not as trivial as it seems since $\Theta_h$ is calculated using the bilinear form $a(\cdot,\cdot)$ and the inner product in $V$.
Recall that $a(\cdot,\cdot)$ is not bounded independently of $k$ and that $\ip\cdot\cdot_{V,k}$ includes terms weighted with negative powers
of the time step size $k$. 
\begin{lemma}\label{lem:b}
  The bilinear form $b(\cdot,\Theta_h\cdot)$ is bounded,
  \begin{align*}
    |b(\uu,\Theta_h\ww_h)| &\lesssim \norm{\uu}{U,k}\norm{\ww_h}{U,k},\\
    |b(\uu,\Theta_h\ww_h)| &\lesssim \norm{\uu}{U,k}\enorm{\ww_h}_k,
  \end{align*}
for all $\uu\in U$, $\ww_h\in U_h$, and fulfills the inf-sup condition
\begin{align*}
  \enorm{\uu_h}_k \lesssim \sup_{0\neq \vv_h\in U_h} \frac{b(\uu_h,\Theta_h\vv_h)}{\enorm{\vv_h}_k}
\end{align*}
for all $\uu_h\in U_h$.
In particular, we conclude that Problem~\ref{eq:ellipticProjection} admits a unique solution. 
\end{lemma}
\begin{proof}
  We show only the second boundedness estimate, as the first one follows from $\enorm\cdot_k\lesssim \norm{\cdot}{U,k}$.
  To that end let $v:= \Theta_h\ww_h=\Theta_h(w_h,\widehat\chi_h)$. 
  We use the splitting $v = w_h + \widetilde w_h$. 
  From the definition of the optimal test function we see that
  \begin{align*}
    &\frac1{k}\ip{w_h}{\delta v} + \ip{\AA\nabla w_h}{\pwnabla \delta v} 
    + \frac1{k}\ip{\widetilde w_h}{\delta v} + \ip{\AA\pwnabla \widetilde w_h}{\pwnabla \delta v} 
    \\
    &\qquad = \frac1{k}\ip{w_h}{\delta v} + \ip{\AA\nabla w_h}{\pwnabla\delta v} + \ip{\bbeta\cdot\nabla w_h}{\delta v} + \ip{\gamma w_h}{\delta v} - \dual{\widehat\chi_h}{\delta v}_{\cS}.
  \end{align*}
  Therefore the component $\widetilde w_h\in V_h$ satisfies
  \begin{align*}
    \frac1{k}\ip{\widetilde w_h}{\delta v} + \ip{\AA\pwnabla \widetilde w_h}{\pwnabla \delta v} 
    = \ip{\bbeta\cdot\nabla w_h}{\delta v} + \ip{\gamma w_h}{\delta v} - \dual{\widehat\chi_h}{\delta v}_{\cS}
    \quad\forall \delta v\in V_h.
  \end{align*}
  We note that $a( (0,\widehat\chi_h),\delta v) = -\dual{\widehat\chi_h}{\delta v}_{\cS}$
  and therefore the definition of the optimal test function yields 
  \begin{align*}
    |\dual{\widehat\chi_h}{\delta v}_{\cS}| \leq \norm{\Theta_h(0,\widehat\chi_h)}{V,k}\norm{\delta v}{V,k}.
  \end{align*}
  Standard estimates then show that
  \begin{align*}
    \norm{\widetilde w_h}{V,k}^2 \leq k^{1/2}\norm{\bbeta\cdot\nabla w_h+\gamma w_h}{}k^{-1/2}\norm{\widetilde w_h}{} +
    \norm{\Theta_h(0,\widehat\chi_h)}{V,k}\norm{\widetilde w_h}{V,k}.
  \end{align*}
  Using that $k\leq T$ we infer that 
  \begin{align}\label{eq:wTildeK}
    \norm{\widetilde w_h}{V,k}\lesssim k^{1/2}\norm{\nabla w_h}{} + \norm{\Theta_h(0,\widehat\chi_h)}{V,k} \lesssim \enorm{\ww_h}_k.
  \end{align}
  Observe that $\dual{\widehat\sigma_h}{w_h}_{\cS}=0$ since $w_h\in H_0^1(\Omega)$. Finally,
  \begin{align*}
    |b(\uu,v)| &\leq |b(\uu,w_h)| + |b(\uu,\widetilde w_h)| 
    \\ &\lesssim |\ip{\AA\nabla u}{\nabla w_h} + \ip{\bbeta\cdot\nabla u+\gamma u}{w_h}| + \norm{\uu}{U,k} \norm{\widetilde
    w_h}{V,k} \lesssim \norm{\uu}{U,k}\enorm{\ww_h}_k.
  \end{align*}
  This shows boundedness.
  In order to show the inf-sup condition, we first establish the coercivity estimate
  \begin{align}\label{eq:coercB}
    \norm{\AA^{1/2}\nabla u_h}{}^2 \leq 2 b(\uu_h,\Theta_h\uu_h).
  \end{align}
  Recall that the optimal test function is characterised by
  \begin{align*}
    \frac1k\ip{\Theta_h\uu_h}{\delta v} + \ip{\AA\pwnabla \Theta_h\uu_h}{\pwnabla \delta v} = a(\uu_h,\delta v) = \frac1k\ip{u_h}{\delta v} + b(\uu_h,\delta v)
    \quad\forall \delta v\in V_h.
  \end{align*}
  Setting $\delta v = v = \Theta_h\uu_h$ the last identity together with Young's inequality proves that
  \begin{align}\label{eq:ellipticProjection:proof:a}
    b(\uu_h,v) = \frac1{k}\norm{v}{}^2 + \norm{\AA^{1/2}\pwnabla v}{}^2 - \frac1{k}\ip{u_h}{v}
    \geq \frac1{k}\norm{v}{}^2 + \norm{\AA^{1/2}\pwnabla v}{}^2 - \frac1{2k} \norm{u_h}{}^2 - \frac1{2k} \norm{v}{}^2.
  \end{align}
  We are going to estimate $k^{-1/2}\norm{u_h}{}$: 
  From Lemma~\ref{lem:euler} we get that $k^{-1}\norm{u_h}{}^2 + \norm{\AA^{1/2}\nabla u_h}{}^2 \leq k^{-1}\norm{v}{}^2 + \norm{\AA^{1/2}\pwnabla v}{}^2$, hence,
  \begin{align*}
    \norm{\AA^{1/2}\nabla u_h}{}^2 \leq \norm{\AA^{1/2}\pwnabla v}{}^2 + \frac{1}k \norm{v}{}^2 - \frac1{k}\norm{u_h}{}^2.
  \end{align*}
  Combining this with~\eqref{eq:ellipticProjection:proof:a} yields
  \begin{align*}
    b(\uu_h,v) &\geq \frac1{k}\norm{v}{}^2 + \norm{\AA^{1/2}\pwnabla v}{}^2 - \frac1{2k} \norm{u_h}{}^2 - \frac1{2k} \norm{v}{}^2
    \\
    &= \left(\frac1{2}\norm{\AA^{1/2}\pwnabla v}{}^2 + \frac1{2k} \norm{v}{}^2 - \frac1{2k} \norm{u_h}{}^2\right) + \frac1{2}\norm{\AA^{1/2}\pwnabla v}{}^2
    \\
    &\geq \frac1{2}\norm{\AA^{1/2}\pwnabla v}{}^2 + \frac1{2}\norm{\AA^{1/2}\nabla u_h}{}^2 \geq \frac1{2}\norm{\AA^{1/2}\nabla u_h}{}^2.
  \end{align*}
  Now we are in position to establish the inf-sup condition
  \begin{align*}
    \enorm{\uu_h}_k \lesssim \sup_{0\neq \vv_h\in U_h} \frac{b(\uu_h,\Theta_h\vv_h)}{\enorm{\vv_h}_k}:
  \end{align*}
  With the boundedness estimates for $b(\cdot,\cdot)$ we have that
  \begin{align*}
     \norm{\Theta_h(0,\widehat\sigma_h)}{V,k}^2 
    &= a( (0,\widehat\sigma_h),\Theta_h(0,\widehat\sigma_h)) = b( (0,\widehat\sigma_h),\Theta_h(0,\widehat\sigma_h))
    \\ 
    &= b(\uu_h,\Theta_h(0,\widehat\sigma_h)) - b( (u_h,0),\Theta_h(0,\widehat\sigma_h))
    \\
    &\lesssim \sup_{0\neq \vv_h\in U_h} \frac{b(\uu_h,\Theta_h\vv_h)}{\enorm{\vv_h}_k} \norm{\Theta_h(0,\widehat\sigma_h)}{V,k} 
    + \norm{\AA^{1/2}\nabla u_h}{} \norm{\Theta_h(0,\widehat\sigma_h)}{V,k}.
  \end{align*}
  Applying the coercivity estimate~\eqref{eq:coercB} for the second term yields
  \begin{align*}
    \norm{\AA^{1/2}\nabla u_h}{}^2\lesssim b(\uu_h,\Theta_h\uu_h) = \frac{b(\uu_h,\Theta_h\uu_h)}{\enorm{\uu_h}_k}
    \enorm{\uu_h}_k.
  \end{align*}
  Combining the latter two estimates shows that
  \begin{align*}
    \enorm{\uu_h}_k^2 \lesssim \left(\sup_{0\neq \vv_h\in U_h} \frac{b(\uu_h,\Theta_h\vv_h)}{\enorm{\vv_h}_k}\right)^2 +
    \sup_{0\neq \vv_h\in U_h} \frac{b(\uu_h,\Theta_h\vv_h)}{\enorm{\vv_h}_k}\enorm{\uu_h}_k.
  \end{align*}
  Young's inequality finishes the proof of the $\inf$--$\sup$ condition. 
\end{proof}

\begin{lemma}\label{lem:ellipticProjection}
  Let $\EE_h\uu:=(u_h,\widehat\sigma_h)\in U_h$ be the solution of Problem~\eqref{eq:ellipticProjection}.
  Then it holds
  \begin{align*}
    \norm{\nabla(u-u_h)}{} \leq \enorm{\uu-\uu_h}_k \lesssim \inf_{\vv_h\in U_h} \norm{\uu-\vv_h}{U,k}. 
  \end{align*}
  
  If additionally $hk^{-1/2}=\OO(1)$ (see Lemma~\ref{lem:euler}) then
  \begin{align*}
    \norm{\widehat\sigma-\widehat\sigma_h}{-1/2,k} \lesssim \inf_{\vv_h\in U_h} \norm{\uu-\vv_h}{U,k}.
  \end{align*} 
\end{lemma}
\begin{proof}
  The best approximation properties follow from standard arguments. We give the details only for sake of completeness:
  Let $\ww_h\in U_h$ be arbitrary, then
  \begin{align*}
    \enorm{\uu-\uu_h}_k \leq \enorm{\uu-\ww_h}_k + \enorm{\uu_h-\ww_h}_k,
  \end{align*}
  and using~\eqref{eq:ellipticProjection} and the preceding lemma,
  \begin{align*}
    \enorm{\uu_h-\ww_h}_k 
    \lesssim \sup_{0\neq \vv_h\in U_h} \frac{b(\uu_h-\ww_h,\Theta_h\vv_h)}{\enorm{\vv_h}_k}
    \leq \sup_{0\neq \vv_h\in U_h} \frac{b(\uu-\ww_h,\Theta_h\vv_h)}{\enorm{\vv_h}_k}\lesssim \norm{\uu-\ww_h}{U,k}.
  \end{align*}
  
  Since $\enorm{\uu}_k\gtrsim \norm{\nabla u}{}$ this shows the first of the claimed best-approximation estimates. 
  The second follows by noting that $\norm{\widehat\sigma_h}{-1/2,k}\lesssim \norm{\Theta_h(0,\widehat\sigma_h)}{V,k}$ if $hk^{-1/2}=\OO(1)$ (see Lemma~\ref{lem:euler}) and considering
  \begin{align*}
    \norm{\widehat\sigma-\widehat\sigma_h}{-1/2,k}\leq\norm{\uu-\uu_h}{U,k} \leq \norm{\uu-\ww_h}{U,k} +
    \norm{\uu_h-\ww_h}{U,k} \lesssim \norm{\uu-\ww_h}{U,k} + \enorm{\uu_h-\ww_h}_k.
  \end{align*}
  The last term is handled as before which concludes the proof.
\end{proof}

Combining the latter quasi-best approximation result with standard approximation properties yields
\begin{corollary}\label{cor:convRatesEnergyNorm}
Suppose that the components of $\uu = (u,\widehat\sigma)$ are sufficiently smooth, then
  \begin{align*}
    \enorm{\uu-\EE_h\uu}_k  = \OO(h^{p+1}).
  \end{align*}

  If additionally $hk^{-1/2}=\OO(1)$ then
  \begin{align*}
      \norm{\uu-\EE_h\uu}{U,k} = \OO(h^{p+1}).
  \end{align*}
\end{corollary}

\section{Optimal error estimate in energy norm}\label{sec:strong}
This section is devoted to prove optimal error estimates in the $H^1(\Omega)$ norm of the $n$-th solution of the backward Euler method~\eqref{eq:euler}.
We make the same assumptions on the regularity of solutions as in~\cite{thomee}.

\begin{theorem}\label{thm:energyerror}
  Let $\uu_h^n=(u_h^n,\widehat\sigma_h^n)\in U_h$ denote the solution of~\eqref{eq:euler}.
  Suppose that the components of $\uu(t;\cdot)$ are sufficiently regular. 
  Under these assumptions there exists $k_0$ (independent of $h$) such that for $k<k_0$ the solution satisfies
  \begin{align*}
    \norm{\nabla u^n-\nabla u_h^n}{} = \OO(h^{p+1}+k) + \OO(\norm{\nabla(u^0-u_h^0)}{}).
  \end{align*}

  If additionally, $hk^{-1/2}=\OO(1)$ then
  \begin{align*} 
    \norm{\widehat\sigma^n-\widehat\sigma_h^n}{-1/2,k} = \OO(h^{p+1}+k) + \OO(\norm{\nabla(u^0-u_h^0)}{}).
  \end{align*}
\end{theorem}
\begin{proof}
  With the elliptic projection operator $\EE_h$ (see~\eqref{eq:ellipticProjection}) we consider the splitting
  \begin{align*}
    \uu^n-\uu_h^n = (\uu^n-\EE_h\uu^n)+(\EE_h\uu^n-\uu_h^n).
  \end{align*}
  We may also use the norm $\enorm{\cdot}_k$ defined in the proof of Lemma~\ref{lem:ellipticProjection}.

  \noindent
  \textbf{Step 1.}
  By Corollary~\ref{cor:convRatesEnergyNorm} we have
  \begin{align*}
    \enorm{\uu^n-\EE_h\uu^n}_k = \OO(h^{p+1}).
  \end{align*}

  \noindent 
  \textbf{Step 2.}
  We derive error equations: First, write $\EE_h\uu^n = (\EE_h^1\uu^n,\EE_h^2\uu^n)$. 
  Then, by~\eqref{eq:ellipticProjection}
  \begin{align*}
    a(\EE_h\uu^n,v_\mathrm{opt}) &= \frac1{k}\ip{\EE_h^1\uu^n}{v_\mathrm{opt}} + b(\EE_h\uu^n,v_\mathrm{opt}) 
    = \frac1{k}\ip{\EE_h^1\uu^n}{v_\mathrm{opt}} + b(\uu^n,v_\mathrm{opt}) 
    \\
    &= \ip{k^{-1}\EE_h^1\uu^n+f^n-\dot{u}^n}{v_\mathrm{opt}} \quad\forall v_\mathrm{opt} \in \Theta_h(U_h).
  \end{align*}
  Second, by~\eqref{eq:euler}
  \begin{align*}
    a(\uu_h^n,v_\mathrm{opt}) = \ip{f^n+k^{-1}u_h^{n-1}}{v_\mathrm{opt}} \quad\forall v_\mathrm{opt} \in \Theta_h(U_h).
  \end{align*}
  Third, combining both identities and writing $\ww^n = (w^n,\widehat\chi^n):= \EE_h\uu^n-\uu_h^n$ yields
  \begin{align}\label{eq:erroreqOld}
    \frac1{k}\ip{w^n}{v_\mathrm{opt}} + b(\ww^n,v_\mathrm{opt}) = \frac1{k}\ip{e_h^n+w^{n-1}}{v_\mathrm{opt}} \quad\forall v_\mathrm{opt} \in \Theta_h(U_h),
  \end{align}
  where 
  \begin{align}\label{def:ehn}
    e_h^n := \EE_h^1(\uu^n-\uu^{n-1})-k\dot{u}^n.
  \end{align}
  Putting the term with $w^{n-1}$ on the left-hand side yields
  \begin{align}\label{eq:erroreq}
    \frac1{k}\ip{w^n-w^{n-1}}{v_\mathrm{opt}} + b(\ww^n,v_\mathrm{opt}) = \frac1{k}\ip{e_h^n}{v_\mathrm{opt}} \quad\forall v_\mathrm{opt} \in \Theta_h(U_h),
  \end{align}

  \noindent 
  \textbf{Step 3.}
  We use the test function $v_\mathrm{opt}=\Theta_h(v_h,0) = v_h +\widetilde v_h$ with $v_h = w^n-w^{n-1}$. 
  Reordering the terms in the error equations~\eqref{eq:erroreq},
  \begin{align}\label{eq:erreqWithC}
    \begin{split}
      \frac1k \norm{w^n-w^{n-1}}{}^2 + \frac12 \norm{\AA^{1/2}\nabla w^n}{}^2 - &\frac12 \norm{\AA^{1/2}\nabla w^{n-1}}{}^2 +
      \frac12 \norm{\AA^{1/2}\nabla(w^n - w^{n-1})}{}^2\\
      &=
    \frac1{k}\ip{e_h^n}{v_h+\widetilde v_h} - \frac1{k}\ip{w^n-w^{n-1}}{\widetilde v_h} 
    \\ &\qquad - b(\ww^n,\widetilde v_h) -\ip{\bbeta\cdot\nabla w^n}{v_h} - \ip{\gamma w^n}{v_h} \\
    &= r_1+r_2+r_3+r_4+r_5.
    \end{split}
  \end{align}

  \noindent 
  \textbf{Step 4.}
  We estimate the contributions $r_j$ of the right-hand side of the error equation~\eqref{eq:erreqWithC}.
  Recall that $v_h = w^n-w^{n-1}$.
  Throughout we use Young's inequality with a parameter $\delta>0$ and the estimate from~\eqref{eq:wTildeK}, i.e., $k^{-1/2}\norm{\widetilde v_h}{} \leq \norm{\widetilde v_h}{V,k}\lesssim k^{1/2}\norm{\nabla (w^n-w^{n-1})}{}$. 
  First, 
  \begin{align*}
    |r_1| &=  \frac1{k}|\ip{e_h^n}{v_h+\widetilde v_h}| \leq k^{-1/2} \norm{e_h^n}{} k^{-1/2}\norm{v_h}{} 
    + k^{-1/2} \norm{e_h^n}{} k^{-1/2}\norm{\widetilde v_h}{} \\
    &\lesssim  \delta^{-1} k^{-1}\norm{e_h^n}{}^2 + \delta k^{-1}\norm{w^n-w^{n-1}}{}^2 + \delta k\norm{\AA^{1/2}\nabla (w^n-w^{n-1})}{}^2.
  \end{align*}
  Second,
  \begin{align*}
    |r_2| &\lesssim k^{-1/2}\norm{w^n-w^{n-1}}{}k^{-1/2}\norm{\widetilde v_h}{} \leq \delta k^{-1} \norm{w^n-w^{n-1}}{}^2 
    + \delta^{-1}k\norm{\AA^{1/2}\nabla (w^n-w^{n-1})}{}^2.
  \end{align*}
  Third, recall the definition of the (discrete) optimal test function, i.e., 
  \begin{align*}
    |\dual{\widehat \chi^n}{\widetilde v_h}_{\cS}| = |a( (0,\widehat\chi^n),\widetilde v_h)| = 
    |\ip{\Theta_h(0,\widehat\chi^n)}{\widetilde v_h}_{V,k}|
  \end{align*}
  which together with the Cauchy--Schwarz inequality shows that
  \begin{align*}
    |r_3| &= |b(\ww^n,\widetilde v_h)| \lesssim (\norm{\nabla w^n}{}+\norm{\Theta_h(0,\widehat\chi^n)}{V,k})k^{1/2}\norm{\nabla v_h}{}
    \\
    &\leq \norm{\nabla (w^n-w^{n-1})}{}k^{1/2}\norm{\nabla v_h}{} + k^{1/2}\norm{\nabla w^{n-1}}{} \norm{\nabla v_h}{}
    + k^{1/2}\norm{\Theta_h(0,\widehat\chi^n)}{V,k}\norm{\nabla v_h}{}.
  \end{align*}
  We need to estimate $\norm{\Theta_h(0,\widehat\chi^n)}{V,k}$: Using the error equation~\eqref{eq:erroreqOld}
  \begin{align*}
    \frac1k\ip{w^n}{\Theta_h(0,\widehat\chi^n)} + b(\ww^n,\Theta_h(0,\widehat\chi^n)) = \frac1k\ip{e_h^n+w^{n-1}}{\Theta_h(0,\widehat\chi^n)}
  \end{align*}
  and the definition of $\Theta_h$ we get that 
  \begin{align*}
    \underbrace{\frac1{k}\ip{\Theta_h(0,\widehat\chi^n)}{w^n}+\ip{\AA\pwnabla \Theta_h(0,\widehat\chi^n)}{\pwnabla w^n}}_{=-\dual{\widehat\chi^n}{w^n}_{\cS}=0} &+ \ip{\bbeta\cdot\nabla w^n+\gamma w^n}{\Theta_h(0,\widehat\chi^n)} + \norm{\Theta_h(0,\widehat\chi^n)}{V,k}^2
    \\&\qquad= \frac1k\ip{e_h^n+w^{n-1}}{\Theta_h(0,\widehat\chi^n)}.
  \end{align*}
  Using the definition of $\Theta_h$ another time shows that
  \begin{align*}
    \frac1k\ip{w^{n-1}}{\Theta_h(0,\widehat\chi^n)} 
    = -\ip{\AA\pwnabla\Theta_h(0,\widehat\chi^n)}{\nabla w^{n-1}} -
    \underbrace{\dual{\widehat\chi^n}{w^{n-1}}_{\cS}}_{=0} \lesssim \norm{\Theta_h(0,\widehat\chi^n)}{V,k}\norm{\nabla w^{n-1}}{}.
  \end{align*}
  Combining the latter two estimates together with standard estimates gives
  \begin{align*}
    \norm{\Theta_h(0,\widehat\chi^n)}{V,k} \lesssim k^{-1/2}\norm{e_h^n}{} + \norm{\nabla w^{n-1}}{} + k^{1/2}\norm{\nabla w^{n}}{}.
  \end{align*}
  Using the error equations again together with Lemma~\ref{lem:euler} it is easy to see that $k^{1/2}\norm{\nabla w^n}{}\lesssim \norm{e_h^n}{} + \norm{\nabla w^{n-1}}{}$. Therefore, 
  \begin{align}\label{eq:estChiN}
    \norm{\Theta_h(0,\widehat\chi^n)}{V,k} \lesssim k^{-1/2}\norm{e_h^n}{} + \norm{\nabla w^{n-1}}{}.
  \end{align}
  For the term $r_3$ we thus get the estimate
  \begin{align*}
    |r_3| \lesssim (k^{1/2}+\delta) \norm{\AA^{1/2}\nabla(w^n-w^{n-1})}{}^2 + k \delta^{-1} \norm{\nabla w^{n-1}}{}^2 + \delta^{-1}\norm{e_h^n}{}^2.
  \end{align*}
  Fourth, 
  \begin{align*}
    |r_4| & \lesssim k^{1/2}\norm{\nabla w^n}{}k^{-1/2}\norm{v_h}{} 
    \lesssim k \delta^{-1} \norm{\AA^{1/2}\nabla (w^n-w^{n-1})}{} + k\delta^{-1}\norm{\AA^{1/2}\nabla w^{n-1}}{}^2 + \delta k^{-1}\norm{w^n-w^{n-1}}{}^2.
  \end{align*}
  Fifth, using Poincar\'e's inequality,
  \begin{align*}
    |r_5| \lesssim \norm{w^n}{} \norm{v_h}{} &\lesssim k\delta^{-1}\norm{\nabla w^n}{}^2 + \delta k^{-1}\norm{v_h}{}{}^2\\
    & \lesssim k\delta^{-1}\norm{\AA^{1/2}\nabla(w^n - w^{n-1})}{}^2 + k\delta^{-1}\norm{\AA^{1/2}\nabla w^{n-1}}{}^2 + \delta k^{-1}\norm{w^n-w^{n-1}}{}{}^2.
  \end{align*}
  
  \noindent 
  \textbf{Step 5.}
  Using the bounds for the $r_j$ in the error equation~\eqref{eq:erreqWithC}, choosing first $\delta$ sufficiently small independently of $k$,
  and then $k\leq k_0$ sufficiently small in dependence of $\delta$,
  and taking into account the norm equivalence
  $\norm{\AA^{1/2}\nabla(\cdot)}{}\eqsim \norm{\nabla(\cdot)}{}$, we obtain
  \begin{align*}
    \norm{\AA^{1/2}\nabla w^n}{}^2 - \norm{\AA^{1/2}\nabla w^{n-1}}{}^2 \lesssim k^{-1}\norm{e_h^n}{}^2 + k \norm{\nabla w^{n-1}}{}^2.
  \end{align*}
  Iterating this estimate,
  applying the discrete Gronwall inequality~\cite[Lemma~10.5]{thomee}, $\sum_{j=1}^n k \leq T$ and norm equivalence $\norm{\AA^{1/2}\nabla(\cdot)}{}\eqsim \norm{\nabla(\cdot)}{}$ once again show that
  \begin{align*}
    \norm{\nabla w^n}{}^2 \lesssim \sum_{j=1}^n \frac1{k} \norm{e_h^j}{}^2 + \norm{\nabla w^0}{}^2.
  \end{align*}

  \noindent 
  \textbf{Step 6.}
  The basic ideas to estimate the sum $\sum_{j=1}^n \frac1{k} \norm{e_h^j}{}^2$ are the same as given 
  in~\cite[Section~1]{thomee} with some modifications, i.e., the use of elliptic projection operators and its properties defined in the present work. 
  For the sake of completeness we recall the main steps where we closely follow the presentation of our own work~\cite{LSQparabolic}:
  We write the error $e_h^j$ defined in~\eqref{def:ehn} as
  \begin{align*}
    e_h^j := e_h^{j,1}+e_h^{j,2} := [\EE_h^1(\uu^j-\uu^{j-1})-(u^j-u^{j-1})] + [u^j-u^{j-1}-k\dot u^j].
  \end{align*}
  We write $\uu^j-\uu^{j-1} = \int_{t_{j-1}}^{t_j} \dot{\uu}(s)\,\mathrm{d}s$. Then, the first term is estimated with Corollary~\ref{cor:convRatesEnergyNorm} and the Cauchy--Schwarz inequality in the time variable, yielding
  \begin{align*}
    \norm{e_h^{j,1}}{}\lesssim \int_{t_{j-1}}^{t_j} h^{p+1} C(\dot{\uu}(s)) \,\mathrm{d}s 
    \leq k^{1/2} \left( \int_{t_{j-1}}^{t_j} h^{2(p+1)} C(\dot{\uu}(s))^2 \,\mathrm{d}s\right)^{1/2}.
  \end{align*}
  Here, $C(\dot{\uu}(s))$ depends on higher-order Sobolev norms of $\dot{\uu}(s)$. 
  Similar arguments together with a Taylor expansion show for the second contribution $e_h^{j,2}$ that
  \begin{align*}
    \norm{e_h^{j,2}}{} = \norm{u^j-u^{j-1}-k\dot u^j}{} \leq k \int_{t_{j-1}}^{t_j} \norm{\ddot u(s)}{} \,\mathrm{d}s
    \leq k^{3/2}\left(\int_{t_{j-1}}^{t_j} \norm{\ddot u(s)}{}^2 \,\mathrm{d}s\right)^{1/2}.
  \end{align*}
  Combining the latter two estimates we conclude that
  \begin{align*}
    \frac1{k}\sum_{j=1}^n \norm{e_h^j}{}^2 = \OO(h^{2(p+1)} + k^2).
  \end{align*}

  \noindent
  \textbf{Step 7.}
  The trace estimate follows from
  \begin{align*}
    \norm{\Theta(0,\widehat\chi^n)}{V,k} \lesssim k^{-1/2}\norm{e_h^n}{} + \norm{\nabla w^{n-1}}{},
  \end{align*}
  cf.~\eqref{eq:estChiN}, by norm equivalence $\norm{\Theta(0,\widehat\chi^n)}{V,k}\eqsim \norm{\widehat\chi^n}{-1/2,k}$ (Lemma~\ref{lem:euler}) under the assumption $hk^{-1/2}=\OO(1)$.
\end{proof}

\section{Optimal error estimate in weaker norms}\label{sec:weak}

Throughout this section we assume that the coefficients $\AA$, $\bbeta$, $\gamma$ and the domain $\Omega$ are such that 
for given data $f,g\in L^2(\Omega)$ the unique solutions $w,v\in H_0^1(\Omega)$ of the primal problem
\begin{subequations}\label{eq:primalForm}
\begin{align}
  -\div\AA\nabla w + \bbeta\cdot\nabla w + \gamma w &= f \quad\text{in }\Omega, \\
  w|_{\Gamma} &= 0,
\end{align}
\end{subequations}
and the dual problem
\begin{subequations}\label{eq:dualForm}
\begin{align}
  -\div(\AA\nabla v + \bbeta v) + \gamma v &= g \quad\text{in }\Omega, \\
  v|_{\Gamma} &= 0,
\end{align}
\end{subequations}
satisfy the regularity estimates
\begin{align}\label{eq:regest}
  \norm{w}{H^2(\Omega)} + \norm{\AA\nabla\ww}{H^1(\TT)}\lesssim \norm{f}{}, \qquad
  \norm{v}{H^2(\Omega)}\lesssim \norm{g}{}. 
\end{align}

\subsection{Further analysis of elliptic projection}
We need to take a closer look at the elliptic projection operator in order to show higher rates in the $L^2(\Omega)$ norm.
To that end we will consider duality arguments (Aubin--Nitsche trick), see~\cite{BoumaGH_DPGconvRates} for the primal DPG method for the Poisson problem. 
Duality arguments for the DPG method with ultra-weak formulation have been explored in~\cite{SupconvDPG,SupconvDPG19}. 
Contrary to the mentioned works, here we need to develop duality arguments for a non-symmetric bilinear form that defines the elliptic projection. Additionally, the norms in this work depend on the time step which can be arbitrarily small, making the analysis more delicate.

Recall the definition of the elliptic projection from~\eqref{eq:ellipticProjection}: Given $\uu\in U$, define $\EE_h\uu
:= \uu_h\in U_h$ by
\begin{align}\label{eq:ellipticProjectionDef2}
  b(\uu_h,\Theta_h\ww_h) = b(\uu,\Theta_h\ww_h) \quad\forall \ww_h\in U_h, 
\end{align}
where the discrete trial-to-test operator $\Theta_h\colon U_h\to V_h$ is defined through
\begin{align*}
  \ip{\Theta_h\ww_h}{\delta v}_{V,k} = a(\ww_h,\delta v) \quad\forall \delta v\in V_h.
\end{align*}

For the analysis we will use an equivalent representation as a mixed system:
\begin{lemma}\label{lem:mixed}
  Problem~\eqref{eq:ellipticProjectionDef2} is equivalent to the mixed system: Find $(v_h,\uu_h)\in V_h\times U_h$ such that
  \begin{subequations}\label{eq:mixed}
  \begin{align}
    \ip{v_h}{\delta v}_{V,k} + b(\uu_h,\delta v) &= b(\uu,\delta v), \label{eq:mixed:a}\\
    a(\delta\ww,v_h) &= 0
  \end{align}
  \end{subequations}
  for all $(\delta v,\delta\ww)\in V_h\times U_h$. 

  In particular,
  \begin{align*}
    \norm{v_h}{V,k}\lesssim \enorm{\uu-\uu_h}_{k}.
 \end{align*}
\end{lemma}
\begin{proof}
  Let $\uu_h\in U_h$ be the solution of~\eqref{eq:ellipticProjectionDef2}. 
  Define $v_h\in V_h$ by $\ip{v_h}{\delta v}_{V,k} = b(\uu-\uu_h,\delta v)$ for all $\delta v \in V_h$. Then,
  \begin{align*}
    a(\delta\ww,v_h) = \ip{\Theta_h\delta\ww}{v_h}_{V,k} = b(\uu-\uu_h,\Theta_h\delta\ww) = 0
  \end{align*}
  for all $\delta\ww\in U_h$. 

  To see the other direction, suppose that $(v_h,\uu_h)\in V_h\times U_h$ solves~\eqref{eq:mixed}.
  With the same consideration as above we get that
  \begin{align*}
    0 = a(\delta\ww,v_h) = \ip{\Theta_h\delta\ww}{v_h}_{V,k} = b(\uu-\uu_h,\Theta_h\delta\ww),
  \end{align*}
  for all $\delta\ww\in U_h$ which means that $\uu_h\in U_h$ satisfies~\eqref{eq:ellipticProjectionDef2}. 

  The final estimate follows from~\eqref{eq:mixed:a} and the definition of the norm $\enorm\cdot_k$,
  \begin{align*}
    \norm{v_h}{V,k}^2 &= b(\uu-\uu_h,v_h) = b( (u-u_h,0),v_h) - \dual{\widehat\sigma-\widehat\sigma_h}{v_h}_{\cS}
    = b( (u-u_h,0),v_h) + a( (0,\widehat\sigma-\widehat\sigma_h),v_h) 
    \\
    &= b( (u-u_h,0),v_h) + \ip{\Theta_h(0,\widehat\sigma-\widehat\sigma_h)}{v_h}_{V,k} \lesssim \enorm{\uu-\uu_h}_k
    \norm{v_h}{V,k}.
  \end{align*}
  This concludes the proof.
\end{proof}

\begin{theorem}\label{thm:duality}
  Suppose that problems~\eqref{eq:primalForm},~\eqref{eq:dualForm} satisfy the regularity
  estimates in~\eqref{eq:regest}. Then, for given $\uu=(u,\widehat\sigma)\in U$, the elliptic projection
  $\uu_h:=(u_h,\widehat\sigma_h):=\EE_h\uu\in U_h$ converges at a higher rate in the $L^2(\Omega)$ norm, i.e.,
  \begin{align*}
    \norm{u-u_h}{} \lesssim h \enorm{\uu-\uu_h}_{k}.
  \end{align*}
\end{theorem}
\begin{proof}
  \textbf{Step 1.}
  Let $v\in H_0^1(\Omega)$ denote the solution of the dual problem~\eqref{eq:dualForm} with $g= u-u_h$.
  Integration by parts and the definition of the bilinear form $b(\cdot,\cdot)$ yield
  \begin{align*}
    \norm{u-u_h}{}^2 &= \ip{u-u_h}{-\div(\AA\nabla v + \bbeta v) + \gamma v} 
    \\ &= \ip{\AA\nabla(u-u_h)}{\nabla v} +
    \ip{\bbeta\cdot\nabla(u-u_h)+\gamma (u-u_h)}{v} = b(\uu-\uu_h,v). 
  \end{align*}

  \noindent
  \textbf{Step 2.} 
  We characterise $v=\Theta \ww$ where $\Theta\colon U\to V$ is the continuous trial-to-test operator defined through
  \begin{align*}
    \ip{\Theta \ww}{\delta v}_{V,k} = a(\ww,\delta v) \quad\forall \delta v \in V. 
  \end{align*}
  From the latter definition we obtain
  \begin{align*}
    \ip{v}{\delta v}_{V,k} = \frac{1}{k} \ip{v}{\delta v} + \ip{\AA\nabla v}{\nabla_\TT \delta v} = a((v,0),\delta v) -
    \ip{\bbeta\cdot\nabla v+\gamma v}{\delta v} \quad\forall \delta v\in V.
  \end{align*}
  It follows that $\ww = (v,0) + \ww^*$ where $\ww^* = (w^*,\widehat\chi^*)\in U$. The component $w^*\in
  H_0^1(\Omega)$ solves
  \begin{align}\label{eq:wStar}
    -\div\AA\nabla w^* + \bbeta\cdot\nabla w^* + \gamma w^* + \frac1{k} w^* = -\bbeta\cdot\nabla v - \gamma v
  \end{align}
  and $\widehat\chi^*$ is given by $\widehat\chi^*|_{\partial T} = \AA\nabla w^*\cdot\normal_K|_{\partial K}$ for all $K\in\TT$.

  \noindent
  \textbf{Step 3.} 
  Let $v_h\in V_h$ be such that $(v_h,\uu_h)\in V_h\times U_h$ is the solution of the equivalent mixed
  system~\eqref{eq:mixed}. Then, we see with Step~1 and using~\eqref{eq:mixed} that
  \begin{align}
    \norm{u-u_h}{}^2 &= b(\uu-\uu_h,v) = b(\uu-\uu_h,v) - \ip{v}{v_h}_{V,k} + a(\ww,v_h) 
    \\ &= b(\uu-\uu_h,v-\delta v) - \ip{v-\delta v}{v_h}_{V,k} + a(\ww-\delta \ww,v_h)
  \end{align}
  for all $(\delta v,\delta\ww)\in V_h\times U_h$. 

  We choose $\delta v \in \PP^{p+1}(\Omega)\cap H_0^1(\Omega)$ to be the best approximation of $v$ in the
  $H_0^1(\Omega)$ norm and $\delta\ww = (\delta v,0) + \delta\ww^*$ where $\delta\ww^*\in U_h$ will be chosen below.
  The first term is estimated (note that $v-\delta v \in H_0^1(\Omega)$) using the regularity of $v$ by
  \begin{align*}
    |b(\uu-\uu_h,v-\delta v)| \lesssim \norm{\nabla(u-u_h)}{} \norm{\nabla(v-\delta v)}{} \lesssim
    h \norm{v}{H^2(\Omega)} \norm{\nabla(u-u_h)}{} \lesssim h \norm{u-u_h}{} \enorm{\uu-\uu_h}_{k}. 
  \end{align*}
  For the remaining terms, note that $\ww-\delta\ww = (v-\delta v,0) + (\ww^*-\delta\ww^*)$.
  Then, using the estimate $\norm{v_h}{V,k}\lesssim \enorm{\uu-\uu_h}_{k}$ from Lemma~\ref{lem:mixed}
  and boundedness~\eqref{eq:a:bdd} of $a$,
  \begin{align*}
    &-\ip{v-\delta v}{v_h}_{V,k} + a(\ww-\delta\ww,v_h) 
    \\&\qquad\qquad= \ip{\bbeta\cdot\nabla(v-\delta v)+\gamma(v-\delta v)}{v_h} + a(\ww^*-\delta\ww^*,v_h)
    \\
    &\qquad\qquad\lesssim k^{1/2}\norm{\nabla(v-\delta v)}{}\norm{v_h}{V,k} + (k^{-1/2}\norm{w^*-\delta w^*}{}+\norm{\ww^*-\delta\ww^*}{U,k})\norm{v_h}{V,k}
    \\
    &\qquad\qquad\lesssim h \norm{u-u_h}{}\norm{\uu-\uu_h}{U,k} + (k^{-1/2}\norm{w^*-\delta w^*}{}+\norm{\ww^*-\delta\ww^*}{U,k})
    \enorm{\uu-\uu_h}_{k}.
  \end{align*}
  To finish the proof it remains to show the estimate $(k^{-1/2}\norm{w^*-\delta w^*}{}+\norm{\ww^*-\delta\ww^*}{U,k})\lesssim
  h\norm{u-u_h}{}$. Basically, this follows from the weak form of the PDE defined in Step~2:
  Let $\delta \ww^*=(\delta w^*,\delta\widehat\chi^*)\in U_h$ with $\delta w^*$ being the best approximation of $w^*$ with respect to $\norm{\nabla(\cdot)}{}$ and
  $\delta \widehat\chi^*|_{\partial K} = \delta \cchi^*\cdot \normal_K|_{\partial K}$ for all $K\in\TT$. 
  Here, $\delta\cchi^*$ denotes the Raviart--Thomas projection of $\cchi^*=\AA\nabla w^*$ into the Raviart--Thomas space of order $p$. 

  The weak form of PDE~\eqref{eq:wStar} shows that
  \begin{align*}
    \frac1{k}\norm{w^*}{}^2 + \norm{\AA^{1/2} w^*}{}^2 &\leq \frac1{k}\ip{w^*}{w^*} + \ip{\AA\nabla w^*}{\nabla w^*} +
    \ip{\bbeta\cdot\nabla w^*+\gamma w^*}{w^*} \\
    &= \ip{-\bbeta\cdot\nabla v-\gamma v}{w^*} \lesssim \norm{\nabla v}{} \norm{w^*}{}
  \end{align*}
  which implies first $k^{-1}\norm{w^*}{}\lesssim \norm{\nabla v}{}$. Bootstrapping this estimate then shows
  \begin{align*}
    k^{-1}\norm{w^*}{} + k^{-1/2}\norm{\nabla w^*}{} \lesssim \norm{\nabla v}{}\lesssim \norm{u-u_h}{}.
  \end{align*}
  Moreover, PDE~\eqref{eq:wStar} together with our regularity assumptions~\eqref{eq:regest}, and using the latter estimate show that
  \begin{align*}
    \norm{w^*}{H^2(\Omega)} \lesssim \norm{-\bbeta\cdot\nabla v-\gamma v-k^{-1}w^*}{} \lesssim \norm{\nabla v}{}\lesssim \norm{u-u_h}{}.
  \end{align*}
  Therefore, standard approximation results and the aforegoing stability analysis prove that
  \begin{align*}
    k^{-1/2}\norm{w^*-\delta w^*}{} + \norm{\nabla(w^*-\delta w^*)}{} \lesssim hk^{-1/2}\norm{\nabla w^*}{} + h\norm{w^*}{H^2(\Omega)}
    \lesssim h \norm{u-u_h}{}.
  \end{align*}
  By our regularity assumptions we have that $\cchi^*=\AA\nabla w^*\in H^1(\TT)^d\cap \Hdivset\Omega$ so that the trace terms can be estimated in a similar fashion as in~\cite[Theorem~5]{SupconvDPG}: From Lemma~\ref{lem:traceidentity} we know that
  \begin{align*}
    \norm{\widehat\chi^*-\delta\widehat\chi^*}{-1/2,k} \eqsim \sup_{0\neq \widetilde v\in V} \frac{\dual{\widehat\chi^*-\delta\widehat\chi^*}{\widetilde v}_{\cS}}{\norm{\widetilde v}{V,k}}.
  \end{align*}
  Then, integration by parts yields
  \begin{align*}
    \dual{\widehat\chi^*-\delta\widehat\chi^*}{\widetilde v}_{\cS} = \ip{\cchi^*-\delta\cchi^*}{\pwnabla \widetilde v} + \ip{\div\cchi^*}{(1-\Pi^p)\widetilde v}
    \lesssim (h \norm{\cchi^*}{H^1(\TT)} + h \norm{\div\cchi^*}{})\norm{\pwnabla \widetilde v}{}
  \end{align*}
  where $\Pi^p$ is the $L^2$ projection on $\PP^p(\TT)$.
  For details on the arguments used we refer to~\cite[Proof of Theorem~5]{SupconvDPG}.
  We stress that the estimates are independent of $k$. 
  Using the bounds from Step~3, this shows that
  \begin{align*}
    \min_{\delta \widehat\chi\in H^{-1/2}(\cS)}\norm{\widehat\chi^*-\delta\widehat\chi}{-1/2,k} \lesssim h \norm{u-u_h}{},
  \end{align*}
  which concludes the proof. 
\end{proof}

\subsection{Error analysis in the $L^2(\Omega)$ norm}

\begin{theorem}
  Let $\uu_h^n\in U_h$ denote the solution of~\eqref{eq:euler}.
  Suppose that the components of $\uu(t;\cdot)$ are sufficiently regular.
  Under the additional regularity assumptions~\eqref{eq:regest} the solution satisfies
  \begin{align*}
    \norm{u^n-u_h^n}{} = \OO(h^{p+2}+k) + \OO(\norm{u^0-u_h^0}{}).
  \end{align*}
\end{theorem}
\begin{proof}
  With the elliptic projection operator $\EE_h$ (see~\eqref{eq:ellipticProjection}) we consider the splitting
  \begin{align*}
    u^n-u_h^n = (u^n-\EE_h^1\uu^n)+(\EE_h^1\uu^n-u_h^n).
  \end{align*}
  By Theorem~\ref{thm:duality} and Corollary~\ref{cor:convRatesEnergyNorm} we get that
  \begin{align*}
    \norm{u^n-\EE_h^1\uu^n}{} = \OO(h^{p+2}).
  \end{align*}
  Writing $\ww^n = (w^n,\widehat\chi^n):= \EE_h\uu^n-\uu_h^n$ we recall the error equation from~\eqref{eq:erroreq},
  \begin{align}
    \frac1{k}\ip{w^n}{v_\mathrm{opt}} + b(\ww^n,v_\mathrm{opt}) = \frac1{k}\ip{e_h^n+w^{n-1}}{v_\mathrm{opt}} \quad\forall v_\mathrm{opt} \in \Theta_h(U_h).
  \end{align}
  We test with $v_\mathrm{opt} = \Theta_h \ww^n$ and using Lemma~\ref{lem:euler} we infer that
  \begin{align*}
    k^{-1}\norm{w^n}{}^2 \leq a(\ww^n,\Theta_h\ww^n) = \frac1{k}\ip{e_h^n+w^{n-1}}{\Theta_h\ww^n} \leq k^{-1/2}\norm{e_h^n+w^{n-1}}{}a(\ww^n,\Theta_h\ww^n)^{1/2},
  \end{align*}
  and further  that
  \begin{align*}
    \norm{w^n}{} \leq \norm{e_h^n}{} + \norm{w^{n-1}}{}.
  \end{align*}
  Iterating the arguments yields
  \begin{align*}
    \norm{w^n}{} \leq \sum_{j=1}^n \norm{e_h^j}{} + \norm{w^0}{}.
  \end{align*}
  The last term is estimated with the triangle inequality together with Theorem~\ref{thm:duality} and Corollary~\ref{cor:convRatesEnergyNorm} to obtain that
  \begin{align*}
    \norm{w^0}{} \leq \norm{u^0-u_h^0}{} + \norm{u^0-\EE_h^1\uu^0}{} = \norm{u^0-u_h^0}{} + \OO(h^{p+2}).
  \end{align*}
  The estimate
  \begin{align*}
    \sum_{j=1}^n \norm{e_h^j}{} = \OO(h^{p+2}+k)
  \end{align*}
  is shown by following the very same argumentation as in~\cite[Theorem~1.5]{thomee}. 
  For the sake of completeness we repeat the main steps which are also similar to the ones aready presented in Step~7 of the proof of Theorem~\ref{thm:energyerror}.
  
  From Step~7 of the proof of Theorem~\ref{thm:energyerror} we recall the splitting
  \begin{align*}
    e_h^j := e_h^{j,1}+e_h^{j,2} := [\EE_h^1(\uu^j-\uu^{j-1})-(u^j-u^{j-1})] + [u^j-u^{j-1}-k\dot u^j].
  \end{align*}
  With the same arguments (without using a Cauchy--Schwarz inequality in the time variable) we get that
  \begin{align*}
    \norm{e_h^{j,2}}{} \leq k \int_{t_{j-1}}^{t_j} \norm{\ddot u(s)}{}\,\mathrm{d}s. 
  \end{align*}
  Similarily, for the first contribution we apply Corollary~\ref{cor:convRatesEnergyNorm} and Theorem~\ref{thm:duality} to get that
  \begin{align*}
    \norm{e_h^{j,1}}{} \lesssim h^{p+2}\int_{t_{j-1}}^{t_j} C(\dot{\uu}(s)) \,\mathrm{d}s,
  \end{align*}
  where $C(\dot{\uu}(s))$ depends on higher-order Sobolev norms of $\dot\uu(s)$.
  Summing over $j$ finishes the proof.
\end{proof}

\bibliographystyle{abbrv}
\bibliography{literature}

\end{document}

%% file: header.tex
\newtheorem{theorem}{Theorem}
\newtheorem{lemma}[theorem]{Lemma}
\newtheorem{corollary}[theorem]{Corollary}

\def\enorm#1{|\hspace*{-.5mm}|\hspace*{-.5mm}|#1|\hspace*{-.5mm}|\hspace*{-.5mm}|}

\newcommand{\ip}[2]{(#1\hspace*{.5mm},#2)}
\newcommand{\dual}[2]{\langle#1\hspace*{.5mm},#2\rangle}

\newcommand{\norm}[3][]{#1\|#2#1\|_{#3}}

\def\div{{\rm div\,}}

\newcommand{\Hdivset}[1]{\boldsymbol{H}(\div;#1)}

\newcommand{\set}[2]{\big\{#1\,:\,#2\big\}}

\newcommand{\pwnabla}{\nabla_\TT}

\newcommand{\R}{\ensuremath{\mathbb{R}}}
\newcommand{\N}{\ensuremath{\mathbb{N}}}

\newcommand{\nn}{\ensuremath{\mathbf{n}}}
\newcommand{\vv}{\ensuremath{\boldsymbol{v}}}
\newcommand{\ww}{\ensuremath{\boldsymbol{w}}}
\newcommand{\TT}{\ensuremath{\mathcal{T}}}

\newcommand{\cS}{\ensuremath{\mathcal{S}}}

\newcommand{\PP}{\ensuremath{\mathcal{P}}}
\newcommand{\OO}{\ensuremath{\mathcal{O}}}
\newcommand{\EE}{\ensuremath{\mathcal{E}}}

\newcommand{\normal}{\ensuremath{{\boldsymbol{n}}}}

\renewcommand{\AA}{\ensuremath{\boldsymbol{A}}}

\newcommand{\bbeta}{\ensuremath{\boldsymbol{\beta}}}

\newcommand{\ssigma}{{\boldsymbol\sigma}}

\newcommand{\cchi}{{\boldsymbol\chi}}

\newcommand{\uu}{\boldsymbol{u}}

\newcounter{constantsnumber}
\def\setc#1{
  \ifthenelse{\equal{#1}{poinc}}{C_{\rm edge}}{ 
   \refstepcounter{constantsnumber}
   \label{const#1}C_{\theconstantsnumber}}}

\def\c#1{
  \ifthenelse{\equal{#1}{poinc}}{C_{\rm edge}}{ 
    C_{\ref{const#1}}}}
